\documentclass[a4paper,11pt]{amsart}
\usepackage{hyperref,latexsym}
\usepackage{enumerate}
\usepackage{graphicx}
\usepackage{color}

\theoremstyle{plain}
\newtheorem{theorem}{Theorem}[section]
\newtheorem{lemma}[theorem]{Lemma}
\newtheorem{corollary}[theorem]{Corollary}
\newtheorem{proposition}[theorem]{Proposition}
\theoremstyle{definition}

\theoremstyle{remark}

\begin{document}

\title[]
      {Dan Reznik's identities and more}

\date{}
\author{Misha Bialy}
\address{School of Mathematical Sciences, Raymond and Beverly Sackler Faculty of Exact Sciences, Tel Aviv University,
Israel} 
\email{bialy@post.tau.ac.il}
\thanks{MB was partially supported by ISF grant 162/15.}

\author{serge Tabachnikov}
\address{Department of Mathematics,
Penn State University,
University Park, PA 16802, USA}
\email{tabachni@math.psu.edu}
\thanks{ST was supported by NSF grant DMS-1510055.}


\begin{abstract}
Dan Reznik found, by computer experimentation, a number of conserved quantities associated with periodic billiard trajectories in ellipses. We prove some of his observations using a non-standard generating function for the billiard ball map. In this way, we also obtain some identities valid for all smooth convex billiard tables.
\end{abstract}

\maketitle


\section{Introduction}
In this note we present an alternative approach to  remarkable conservation laws for families of billiard polygons in ellipses. They were discovered by Dan Reznik \cite{reznik} in his computer experiments and proved in \cite {AST}.
Here we discuss a different method of proofs based on a non-standard generating function for convex billiards discovered in \cite{BM,B}. This approach gives some identities valid for the billiard inside any smooth convex curve, see Theorem \ref{thm1}. We hope that this approach will be useful in the study of other problems on billiards. 

We also prove several other conservation laws  found by Reznik in further abundant computer experiments \cite{reznik2}.

\section{Non-standard generating function and billiard polygons in convex billiards}

Consider  the space of oriented lines in the plane $\mathbb R^2(x_1,x_2)$. A line can be written as 
$$\cos\varphi \cdot x_1+\sin\varphi \cdot x_2=p,$$
where $\varphi$ is the direction of the right normal to the oriented line. Thus  $(p,\varphi) $ are coordinates in the space of oriented lines, see Figure \ref{lines}. The 2-form $\omega=dp \wedge d\varphi$ is the area (symplectic) form on the space of oriented lines used in geometrical optics and integral geometry.

\begin{figure}[h]
	\centering
	\includegraphics[width=0.2\linewidth]{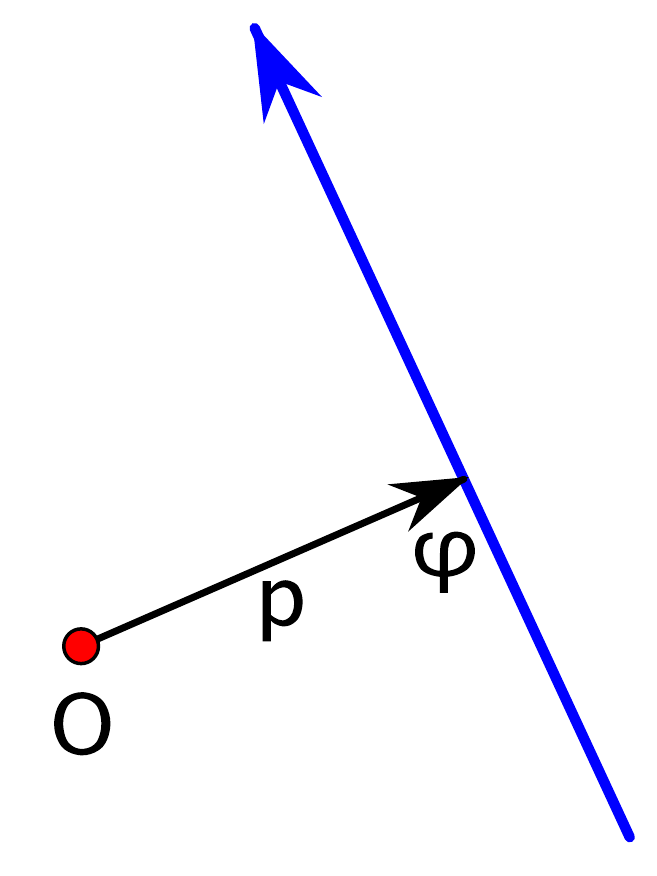}
	\caption{Coordinates in the space of oriented lines.}
	\label{lines}
\end{figure}

Consider a smooth strictly convex billiard curve $\gamma$, and let $h(\varphi)$ be its support function, that is, the signed distance from the origin to the tangent line to $\gamma$ at the point where the outer normal has direction $\varphi$.

The billiard transformation
$$
T: (p_1,\varphi_1)\mapsto (p_2,\varphi_2)
$$
sends the incoming trajectory to the outgoing one. Let  
$$
\psi= \frac{\varphi_1+\varphi_2}{2}, \ \delta=\frac{\varphi_2-\varphi_1}{2},
$$
where $\psi$ is the direction of the outer normal at the reflection point and  $\delta$ is the reflection angle. 

\begin{proposition} \label{prop:genfunct}
The function
$$
S(\varphi_1,\varphi_2)=
2h\left(\frac{\varphi_1+\varphi_2}{2}\right)
\sin\left(\frac{\varphi_2-\varphi_1}{2}\right)=
2h(\psi)
\sin\delta$$
is a generating function of the billiard transformation, that is, $T(p_1,\varphi_1)= (p_2,\varphi_2)$ if and only if
$$ -\frac{\partial S_1(\varphi_1,\varphi_2)}{\partial \varphi_1}=p_1,\quad 
\frac{\partial S_2(\varphi_1,\varphi_2)}{\partial \varphi_2}=p_2.
$$
\end{proposition} 

\begin{figure}[h]
	\centering
	\includegraphics[width=0.5\linewidth]{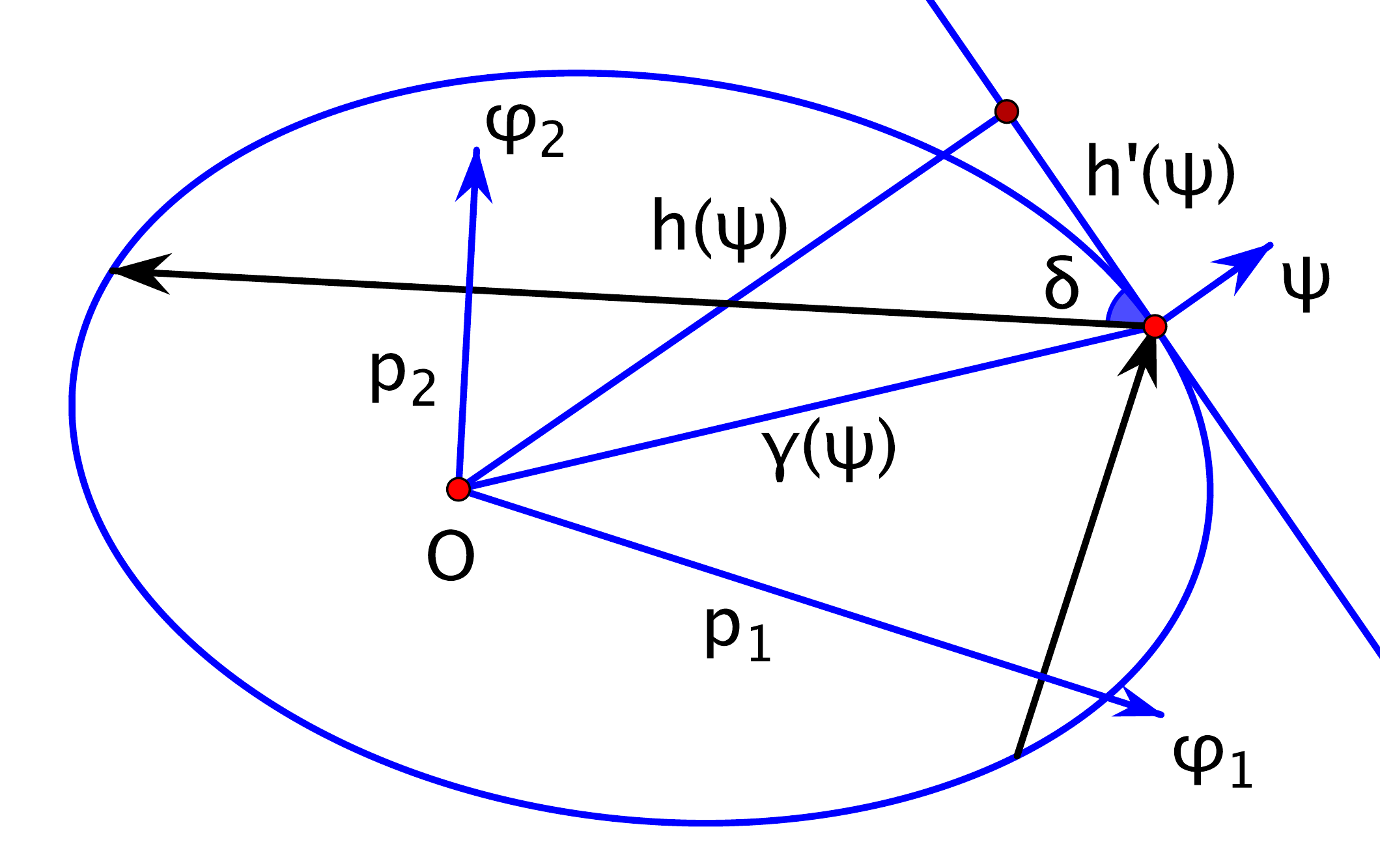}
	\caption{To Proposition \ref{prop:genfunct}.}
	\label{Genfunct}
\end{figure}

\begin{proof} 
We refer to Figure \ref{Genfunct}.

One has
$$
-\frac{\partial S_1(\varphi_1,\varphi_2)}{\partial \varphi_1} = -h'(\psi) \sin \delta + h(\psi) \cos \delta.
$$
The position vector of the point of the curve $\gamma$ with the outer normal having direction $\psi$ is
$$
\gamma(\psi) = h(\psi) (\cos\psi,\sin\psi) + h'(\psi) (-\sin \psi,\cos\psi)
$$
(this formula is well known in convex geometry). Then, using some trigonometry,
$$
p_1 = \gamma(\psi) \cdot (\cos\varphi_1,\sin\varphi_1) = h(\psi) \cos \delta -h'(\psi) \sin \delta,
$$
as needed. The argument for $p_2$ is similar. 
\end{proof}

Let $M_i, i=1,..,n, M_{n+1}=M_1$ be a billiard $n$-gon in $\gamma$. Denote by $\psi_i$ the direction of the outer normal to $\gamma$ at point $M_i$ and by $\delta_i$ the reflection angle at $M_i$. Let $L$ be the perimeter of the $n$-gon.

\begin {theorem}\label{thm1}
Then the following formulas hold:
\begin{enumerate}
\item $\sum_{i=1}^{n} 2h(\psi_i)\sin\delta_i=L$;
\medskip
\item $\sum_{i=1}^{n} h'(\psi_i)\sin\delta_i=0$.
\end{enumerate}
\end{theorem}
\begin{proof} 
	1. We use the approach of  \cite{GM}. 
	
The sum in (1) computes the action of the periodic orbit, that is, the sum of the values of the generating function over the orbit.
 We claim that this sum equals $L$, the action of the periodic orbit with the standard generating function $\mathcal{L}$, the length of a segment of a billiard trajectory (see, e.g., \cite{Ta}). 

 Indeed, consider  the 1-forms $\lambda_1=pd\varphi$ and $\lambda_2=(\cos\delta)ds$ (where $s$ is the arc length parameter on $\gamma$), which  are both primitives of the symplectic form $\omega$ invariant under the billiard transformation $T$. 
 
 Furthermore, $\lambda_1$ and $\lambda_2$ are cohomologous on the phase cylinder. This can be  verified by integrating both forms along the boundary of the phase cylinder: both integrals are equal to the arc length of $\gamma$. Thus  $\lambda_2 -\lambda_1 = dF$ for  some function $F$. Let $\alpha = \lambda_2 -\lambda_1$.
 
One has
$$
T^* \lambda_1 - \lambda_1 = dS,\ T^* \lambda_2 - \lambda_2 = d \mathcal{L},
$$
hence 
$$
d(\mathcal{L}-S)	=T^*\alpha-\alpha=d(F\circ T-F).
$$ 
This implies that
$$
\mathcal{L}-S = F\circ T-F+const.
$$
The constant in the right hand side is zero since both $\mathcal{L}$ and $S$ vanish on the boundary. It remains to note that the sums of $F$ and of $F\circ T$ over a periodic orbit are equal.
\medskip

2. Let the edges of a billiard polygon have coordinates $(p_i,\varphi_i)$, and let 
$$
\psi_i=\frac{\varphi_{i-1}+\varphi_i}{2},\  \delta_i=\frac{\varphi_{i}-\varphi_{i-1}}{2}.
$$
Then 
$$
\frac{\partial S(\varphi_{i-1},\varphi_i)}{\partial \varphi_i}=-\frac{\partial S(\varphi_{i},\varphi_{i+1})}{\partial \varphi_i},
$$
that is, 
\begin{equation}\label{h}
h(\psi_i)\cos\delta_i+ h'(\psi_i)\sin\delta_i=h(\psi_{i+1})\cos\delta_{i+1}- h'(\psi_{i+1})\sin\delta_{i+1}.
\end{equation}
Summing up these equations for $i=1,2,\dots,n$ gives the second statement. 
	\end{proof}


\section{Specializing to ellipses}

Let $\gamma$ be the ellipse $\{\frac{x_1^2}{a_1^2}+\frac{x_2^2}{a_2^2}=1\}$. We will need the support function of $\gamma$, with the origin at the center of the ellipse and the angles $\psi$  made with the positive $x_1$-axis.

\begin{lemma} \label{lemma:supel}
One has: 
$$
h(\psi) = \sqrt{a_1^2 \cos^2\psi + a_2^2 \sin^2\psi}.
$$
\end{lemma}

\begin{proof}
Consider point $(\xi_1,\xi_2)$ of the ellipse. A normal vector is given by
$$
N=\left(\frac{\xi_1}{a_1^2},\frac{\xi_2}{a_2^2}\right) = \ell (\cos \psi,\sin \psi),
$$
and the tangent line at this point has the equation
$$
\frac{\xi_1 x_1}{a_1^2} + \frac{\xi_2 x_2}{a_2^2} =1.
$$
The distance from the origin to this line is
$$
\frac{1}{\sqrt{\frac{\xi_1^2}{a_1^4}+\frac{\xi_2^2}{a_2^4}}} = \frac{1}{\ell}.
$$

On the other hand, 
$$
\xi_1=a_1^2 \ell \cos\psi, \ \xi_2=a_2^2 \ell \sin\psi,
$$
and the equation of the ellipse implies that 
$$
\ell^2 = \frac{1}{a_1^2 \cos^2\psi + a_2^2 \sin^2\psi}.
$$
Therefore $h(\psi) = 1/\ell = \sqrt{a_1^2 \cos^2\psi + a_2^2 \sin^2\psi}$, as claimed.
\end{proof}

The  billiard in ellipse is integrable, and the conserved quantity, called the Joachimsthal integral, is, in the above notation, $|N| \sin \delta$ (see \cite{Ta} or \cite{AST}). This can be written as 
$$
J=\frac{\sin\delta(\psi)}{h(\psi)}.
$$
This means that, along an invariant curve of the billiard, the quantities $h$ and $\sin \delta$ are proportional.

Recall that periodic billiard orbits in ellipses come in 1-parameter families (Poncelet Porism). 
Let $\alpha_i=\pi-2\delta_i$ be the angles of a periodic billiard polygon.
The next statements are corollaries of Theorem \ref{thm1}; the second statement is case $k_{101}$ in \cite{reznik2}.

\begin{corollary} \label{cor:sums}
	For a family of billiard $n$-gons in ellipse, one has
	$$
	2\sum_{i=1}^{n}\sin^2\delta_i =J\cdot L,
	$$ 
	and hence
		$$
	\sum_{i=1}^{n} \cos\alpha_i=J\cdot L -n. \ \     
	$$
One also has
	$$
	2J\sum_{i=1}^{n} h^2(\psi_i)=L.
	$$
	Using the formula for the support function, one has
	$$
	\sum_{i=1}^{n}\cos 2\psi_i=\left(\frac{L}{J}-n(a_1^2+a_2^2)\right)/(a_1^2-a_2^2).
	$$
	The second claim of the Theorem \ref{thm1} implies that
	$$
		\sum_{i=1}^{n} \sin 2\psi_i =0.
	$$
\end{corollary}

Finally, consider a periodic billiard trajectory in an ellipse. The tangent lines at the impact points form a new polygon whose angles are denoted by $\beta_i$, see Figure \ref{exter}. 

\begin{figure}[h]
	\centering
	\includegraphics[width=0.5\linewidth]{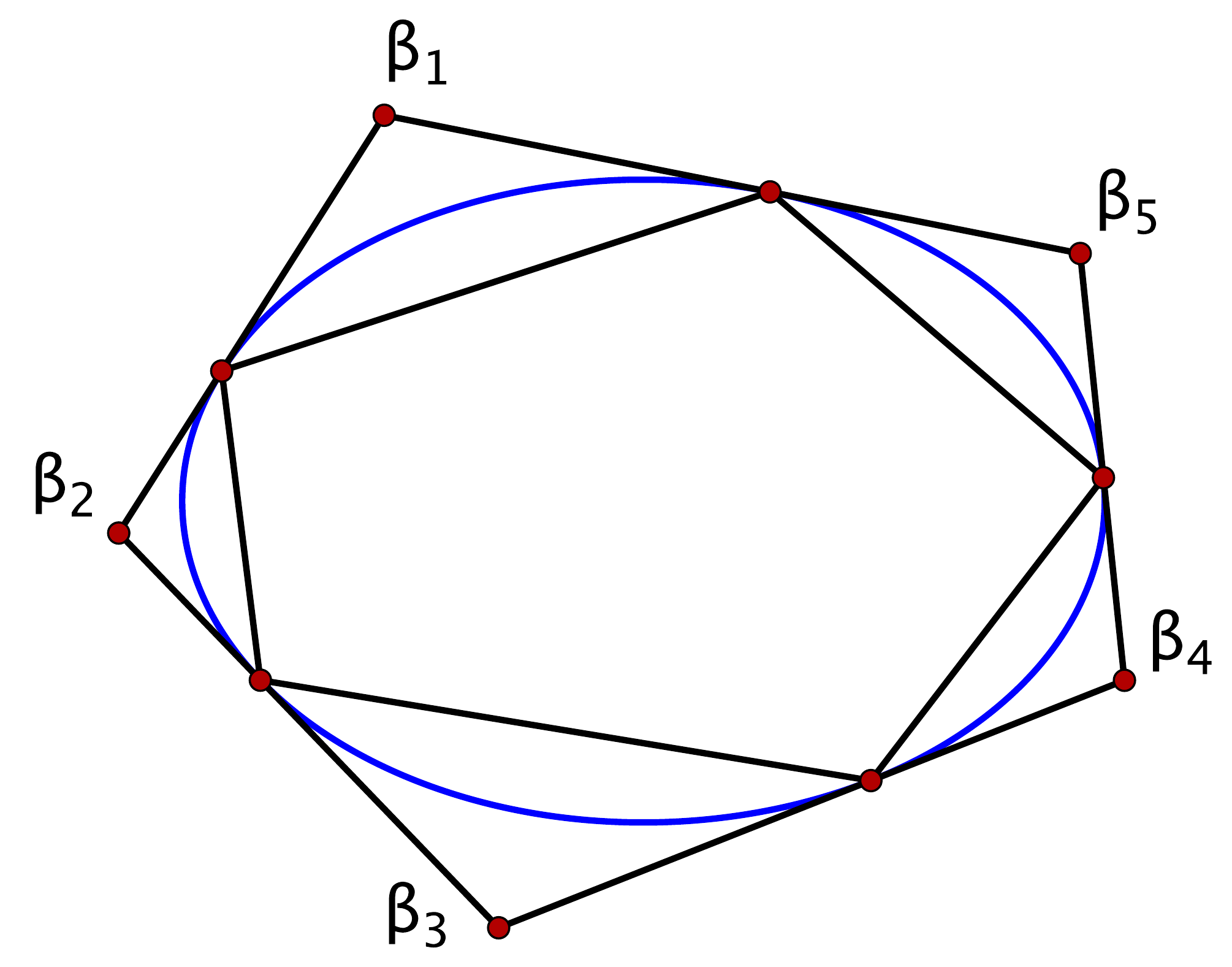}
	\caption{To Theorem \ref{thm:product}.}
	\label{exter}
\end{figure}

 \begin{theorem}[case $k_{102}$ in \cite{reznik2}] \label{thm:product}
	In a 1-parameter family of billiard $n$-gons in ellipse, one has
	$$ \prod_{i=1}^{n}\cos \beta_i= const.
$$
\end{theorem}

\begin{proof}
Note that
$$
\beta_i=\pi-(\psi_{i+1}-\psi_i)=\pi-(\delta_{i+1}+\delta_i).
$$

In the 1-parameter family of billiard $n$-gons circumscribing a confocal ellipse, the reflection angle $\delta$ is a function of the normal direction $\psi$. 
Let us parameterize  the family of $n$-gons by $\psi_1=:\psi$.

	Differentiating the relation
	$$
	\psi_{i+1}-\psi_i=\delta_{i+1}+\delta_i,
	$$ 
	we obtain
	$$
	\frac{d\psi_{i+1}}{d\psi_i}=\frac{1+\delta' (\psi_i)}{1-\delta'(\psi_{i+1})}.
	$$
	
	On the other hand, substituting  
	$$
	J h=\sin\delta, \quad J h'=\cos\delta\cdot\delta'
	$$
	in equation (\ref{h}), 	we arrive at the identity
	$$
	\frac{d\psi_{i+1}}{d\psi_i}=\frac{1+\delta' (\psi_i)}{1-\delta'(\psi_{i+1})}=\frac{\sin 2\delta_{i+1}}{\sin 2\delta_i}.
	$$
	Multiplying these equations implies
	\begin{equation} \label{sin}
	\frac{d\psi_{i}}{d\psi_j}=\frac{\sin 2\delta_{i}}{\sin 2\delta_j}
	\end{equation}
	for not necessarily consecutive $i$ and $j$.
	
	Next we compute the derivative using (\ref{sin}) 
	\begin{equation*}
	\begin{split}
	\frac{d \cos\beta_i}{d\psi}&=-\sin\beta_i \frac{d\beta_i}{d\psi_i}\frac{d\psi_i}{d\psi_1}= -\sin\beta_i \left(1-\frac{\sin 2\delta_{i+1}}{\sin 2\delta_i}\right)\frac{\sin 2\delta_{i}}{\sin 2\delta_1}\\
	&=\frac{\sin\beta_i }{\sin 2\delta_1}\left({\sin 2\delta_{i+1}}-\sin 2\delta_{i}\right)
	= \frac{\sin\beta_i }{\sin 2\delta_1}2\cos(\delta_{i+1}+\delta_i)\sin( \delta_{i+1}-\delta_{i})\\
	&=\frac{\sin\beta_i }{\sin 2\delta_1}2\cos\beta_i\sin( \delta_{i+1}-\delta_{i}).
	\end{split}
	\end{equation*}

		Now we are ready to compute the derivative of the product:
	\begin{equation*}
	\begin{split}
	\frac{d}{d\psi} \prod_{i=1}^{n} \cos \beta_i&=
	\prod_{i=1}^{n} \cos\beta_i \left(\sum_{i=1}^{n} \frac{d \cos\beta_i}{d\psi} \frac{1}{\cos\beta_i}\right)\\
	 &=\prod_{i=1}^{n} \cos\beta_i\left(\sum_{i=1}^{n}{\sin\beta_i }\sin( \delta_{i+1}-\delta_{i})\right)	\frac{2}{\sin 2\delta_1}.
	\end{split}
	\end{equation*}
	It remains to notice that the sum in the parentheses equals zero:
		\begin{equation*}
	\begin{split}
	2\sum_{i=1}^{n}{\sin\beta_i }\sin( \delta_{i+1}-\delta_{i})&=2\sum_{i=1}^{n}{\sin(\delta_{i+1}+\delta_{i})}\sin( \delta_{i+1}-\delta_{i})\\
	&=\sum_{i=1}^{n} (\cos 2\delta_{i}-\cos2\delta_{i+1})=0.
	\end{split}
	\end{equation*}
	This completes the proof.
	\end{proof}
	
Unlike Theorem \ref{thm1}, we do not know the value of the constant in Theorem \ref{thm:product}.	

\section{Further conservation laws}

Let $M_i, i=1,..,n, M_{n+1}=M_1$ be a billiard $n$-gon in an ellipse $\gamma$, and let $\ell_i$ be the tangent line to $\gamma$ at $M_i$. Choose a point $P$, and let $Q_i$ be the foot of the perpendicular dropped from $P$ in $\ell_i$.

\begin{theorem}[cases $k_{306}$ and $k_{302}$ in \cite{reznik2}] \label{thm:cm}
The center of mass of points $Q_i$ and sum $\sum_{i=1}^n |PQ_i|^2$
remain fixed as $M$ varies in the 1-parameter family of $n$-periodic billiard orbits.
\end{theorem}

\begin{proof}
The proof consists of two parts: first,  we show that the statement holds when $P$ is the center of the ellipse, and then we show that it holds for any other point.

We continue to use the notation $h(\psi)$ for the support function of $\gamma$ and $\delta$ for the reflection angle.
The foot point of the perpendicular from the center to the tangent line with  the normal  direction $\psi$ is $(h(\psi) \cos \psi, h(\psi) \sin \psi)$. As we know, up to a multiplicative constant, $h(\psi) = \sin \delta$. 

For the first statement, we  show that
$$
\sum_{i=1}^n \sin \delta_i \cos \psi_i = \sum_{i=1}^n \sin \delta_i \sin \psi_i =0,
$$
that is, the center of mass is at the origin.  Using trigonometry,  this is equivalent to  
$$
\sum_{i=1}^n [\sin(\psi_i-\delta_i) - \sin(\psi_i+\delta_i)] = \sum_{i=1}^n [\cos(\psi_i-\delta_i) - \cos(\psi_i+\delta_i)]  =0.
$$
But $\psi_i\pm \delta_i$ are the normal directions to the consecutive sides of a billiard $n$-gon $M$. Therefore both cyclic sums indeed vanish.

For the second statement, $\sum |PQ_i|^2 = \sum h^2(\psi_i)$, and this is constant by the first statement of Theorem \ref{thm1}.

Next, if point $P$ is translated from the origin to point $(a,b)$, then the support numbers $|PQ_i|$ are changed by $a\cos\psi_i+b\sin\psi_i$,
and 
\begin{equation*} 
\begin{split}
&\sum_{i=1}^n |PQ_i|^2 = \sum_{i=1}^n [h(\psi_i)-(a\cos\psi_i+b\sin\psi_i)]^2 =\\
&\sum_{i=1}^n h^2(\psi_i) - 2 \sum_{i=1}^n h(\psi_i) (a\cos\psi_i+b\sin\psi_i) + \sum_{i=1}^n (a\cos\psi_i+b\sin\psi_i)^2.
\end{split}
\end{equation*} 
Therefore to show that this sum is constant in the 1-parameter family of $n$-periodic orbits, it suffices to show that constant are the individual sums 
$$
\sum_{i=1}^n h(\psi_i) \cos\psi_i,\ \sum_{i=1}^n h(\psi_i) \sin\psi_i,\  \sum_{i=1}^n \cos^2\psi_i,\ \sum_{i=1}^n \cos\psi_i \sin\psi_i,\ \sum_{i=1}^n   \sin^2\psi_i.
$$
As we showed above, the first two sums vanish, and the the remaining three, using some trigonometry, are constant by Corollary \ref{cor:sums}. 

Likewise, when point $P$ is translated from the origin to point $(a,b)$, the feet of the perpendiculars $Q_i$ are translated by the vectors 
$$
[(a,b)\cdot (\sin\psi_i,-\cos\psi_i)] (\sin\psi_i,-\cos\psi_i).
$$
Therefore to show that the center of mass remains constant in the 1-parameter family of $n$-periodic orbits, it suffices to show that constant are the individual sums
$$
\sum_{i=1}^n \cos^2\psi_i,\ \sum_{i=1}^n \cos\psi_i \sin\psi_i,\ \sum_{i=1}^n   \sin^2\psi_i.
$$
This we already know, and this completes the proof.
\end{proof}

As a preparation to the proof of the next theorem, we describe another integral of the billiard inside an ellipse. This integral is known to specialists, but we do not know a reference. 

\begin{lemma} \label{lm:prod}	
The product of the distances from the foci of an ellipse to the segments of a billiard trajectory is an integral of the billiard map.
\end{lemma}	

\begin{proof} We present two arguments, an analytic and a geometric ones. 

Consider a segment of the billiard trajectory tangent to a confocal ellipse $\gamma$ with the semi-axes $b<a$. Let $d_1,d_2$ be the distances from the foci to the segment and $\alpha$ be the direction of its normal. Then 
$$
d_1=h(\alpha)-c\cos\alpha,\  d_2=h(\alpha)+c\cos\alpha,
$$ 
where $h$ is the support function of $\gamma$ and $c^2=a^2-b^2$.
From here, using the formula for the support function from Lemma (\ref{lemma:supel}), we get
$$
\Pi:=d_1d_2=h^2(\alpha)-c^2\cos^2 \alpha=a^2\cos^2\alpha+b^2\sin^2\alpha- c^2 \cos^2 \alpha=b^2.
$$
		
For the geometric argument, consider
Figure \ref{reflection}. The distance from $F_1$ to $AC$ is $|F_1C| \sin (\angle F_1CA)$, and from $F_2$ to $AC$ is $|F_2C| \sin (\angle F_2CA)$. The product of the two is
$$
|F_1C||F_2C| \sin (\angle F_1CA)\sin (\angle F_2CA).
$$ 
Likewise, the product of the distances to $BC$ is
$$
|F_1C||F_2C| \sin (\angle F_1CB)\sin (\angle F_2CB).
$$
It remains to notice that $\angle F_1CA = \angle F_2CB$ and $\angle F_2CA = \angle F_1CB$.
\end{proof}

\begin{figure}[h]
	\centering
	\includegraphics[width=0.4\linewidth]{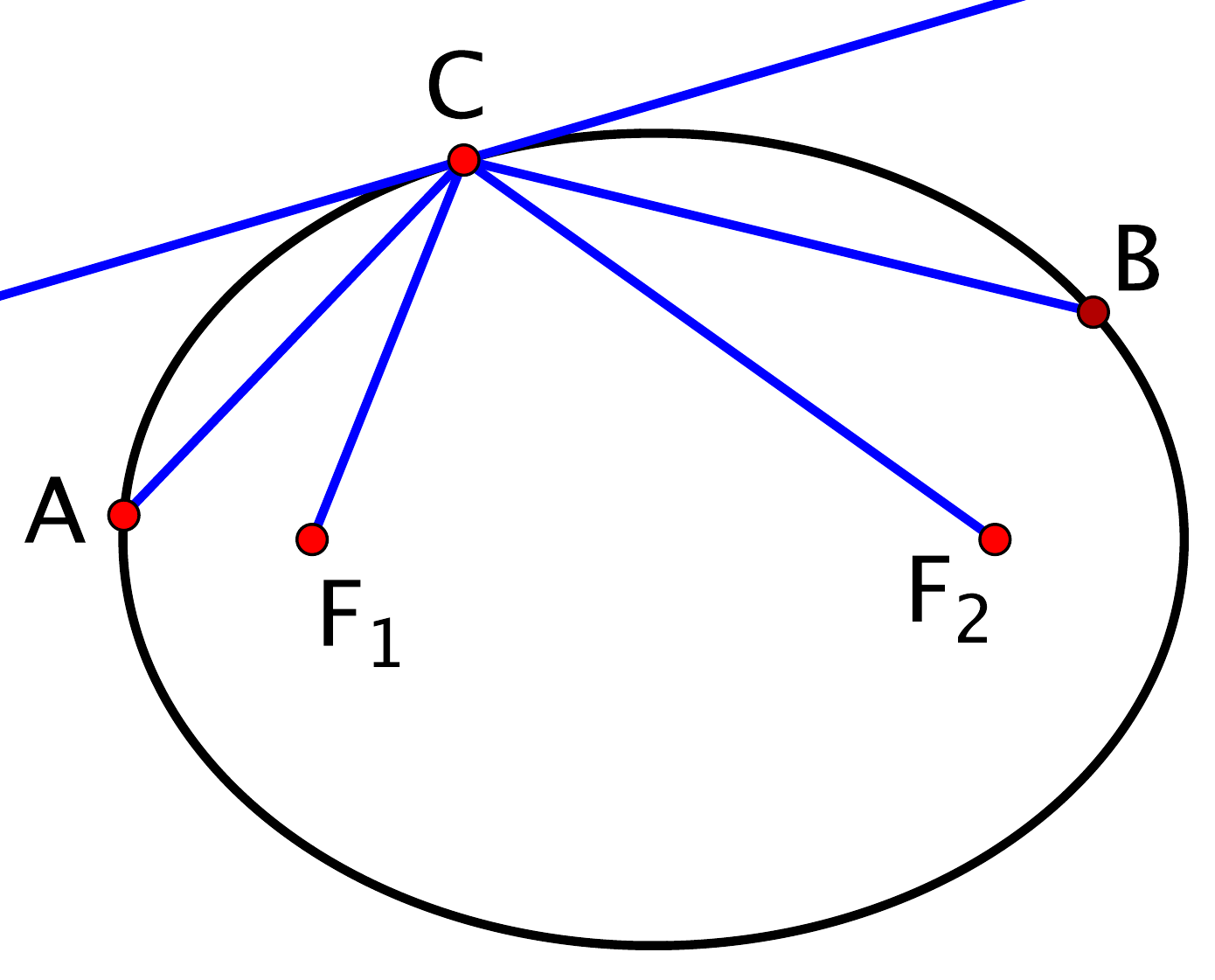}
	\caption{To Lemma \ref{lm:prod}.}
	\label{reflection}
\end{figure}
	
The next results are also among the ones experimentally discovered in \cite{reznik2}.

Let $M_i, i=1,..,n, M_{n+1}=M_1$ be a billiard $n$-gon in an ellipse, circumscribing the confocal ellipse $\gamma$ with the semi-axes $b<a$, and let $R_i$ be the foot of the perpendicular dropped from its center $O$ on the line $M_i M_{i+1}$. Let $F_{1,2}$ be the foci of the ellipse.

\begin{theorem}[cases $k_{202,a}$ and $k_{202,b}$ in \cite{reznik2}] \label{thm:proddist}
(1) If $n$ is even, then $\prod_{i=1}^n |F_1 R_i|$ and $\prod_{i=1}^n |F_2 R_i|$ are constant and
equal to $b^{n}$ as $M$ varies in the 1-parameter family of $n$-periodic billiard orbits.\\
(2) If $n$ is divisible by 4, then $\prod_{i=1}^n |O R_i|$ is constant 
and equal $(ab)^{\frac{n}{2}}$
 as $M$ varies in the 1-parameter family of $n$-periodic billiard orbits.
\end{theorem}	

\begin{proof}
We start from the well known fact that an even-periodic billiard polygon is symmetric with respect to the center of the ellipse. 

For the first statement, let $n=2k$. Then, by symmetry, $|F_1 R_i|=|F_2 R_{i+k}|$ for all $i$. Therefore $\prod_{i=1}^n |F_1 R_i| = \prod_{i=1}^n |F_2 R_i|$. Furthermore, by Lemma \ref{lm:prod}, 
this product equals  $b^{n}$.

For the second statement, let $n=4k$. We claim that it suffices to prove the statement in the case $k=1$. 

Indeed, consider $i$th, $(i+k)$th, $(i+2k)$th, and $(i+3k)$th sides of the periodic $n$-gon. According to the Poncelet Grid Theorem,  \cite{Sch,L-T}, the quadrilateral (in fact, a parallelogram) made by these sides is a 4-periodic billiard trajectory in a confocal ellipse. If we know that the product in question is invariant for $n=4$, then this product for $n=4k$ is $k$th power of the one for $n=4$, and hence invariant as well.

It remains to prove the statement for periodic quadrilaterals.  For want of a geometric argument, we present two analytic proofs.

{\it The first proof.}
A pair of confocal ellipses possessing a Poncelet quadrilateral are given by the formulas
\begin{equation} \label{eq:pair}
\frac{x^2}{a^2}+ \frac{y^2}{b^2}=1,\ \frac{x^2}{A^2}+ \frac{y^2}{B^2}=1
\end{equation}
with $A^2=a^2+ab, B^2=b^2+ab$. These relations become evident if one considers a rectangle that circumscribes the first ellipse and whose sides are parallel to its axes. 

Let $P$ and $Q$ be adjacent vertices of a billiard quadrilateral. Then the tangent lines at $P$ and $Q$ are orthogonal and their intersection point lies on the orthoptic circle centered at $O$, see \cite{C-Z}. This is illustrated in Figure \ref{quad}.

The distance from $O$ to $PQ$ equals $[P,Q]/|P-Q|$. The next vertex after $Q$ is $-P$, therefore we need to show the invariance of
\begin{equation} \label{eq:rat}
\frac{[P,Q]^2}{|P-Q||P+Q|}
\end{equation}
as point $P$ varies.

Let $P=(A\cos\alpha,B\sin\alpha), Q=(A\cos\beta,B\sin\beta)$. Then the normals at points $P$ and $Q$ are given by $(\cos\alpha /A,\sin\alpha /B)$ and $(\cos\beta /A,\sin\beta /B)$. As we mentioned, these normals are orthogonal, hence
$$
\frac{\cos\alpha\cos\beta}{A^2} +\frac{\sin\alpha\sin\beta}{B^2} =0.
$$
Using some trigonometry, we rewrite this as $(A^2+B^2)\cos(\alpha-\beta)=(A^2-B^2)\cos(\alpha+\beta)$, or as
\begin{equation} \label{eq:ort}
(a+b)\cos(\alpha-\beta)=(a-b)\cos(\alpha+\beta).
\end{equation}
This equation describes the relation between points $P$ and $Q$.

Now consider (\ref{eq:rat}). Since $[P,Q]=AB\sin^2(\alpha-\beta)$, the numerator equals $A^2B^2\sin^2(\alpha-\beta)$. To compute the denominator, use the formula
$$
|P-Q||P+Q| = \frac{1}{2} (|P-Q|+|P+Q|)^2 -|P|^2-|Q|^2.
$$
The sum $|P-Q|+|P+Q|$ is the semi-perimeter of the billiard polygon, a conserved quantity equal, in our case, to $2(a+b)$. Hence, again using trigonometry and (\ref{eq:ort}),
\begin{equation*}
\begin{split}
&|P-Q||P+Q| = 2(a+b)^2 - A^2 \cos^2\alpha - B^2\sin^2\alpha - A^2\cos^2\beta - B^2\sin^2\beta=\\
&(a+b)^2 + (b^2-a^2) \cos(\alpha+\beta)\cos(\alpha-\beta) = (a+b)^2 [1-\cos^2(\alpha-\beta)]= \\
&(a+b)^2\sin^2(\alpha-\beta).
\end{split}
\end{equation*}
It follows that (\ref{eq:rat}) equals $ab$ which implies the statement of the theorem. 
\medskip

\begin{figure}[h]
	\centering
	\includegraphics[width=0.4\linewidth]{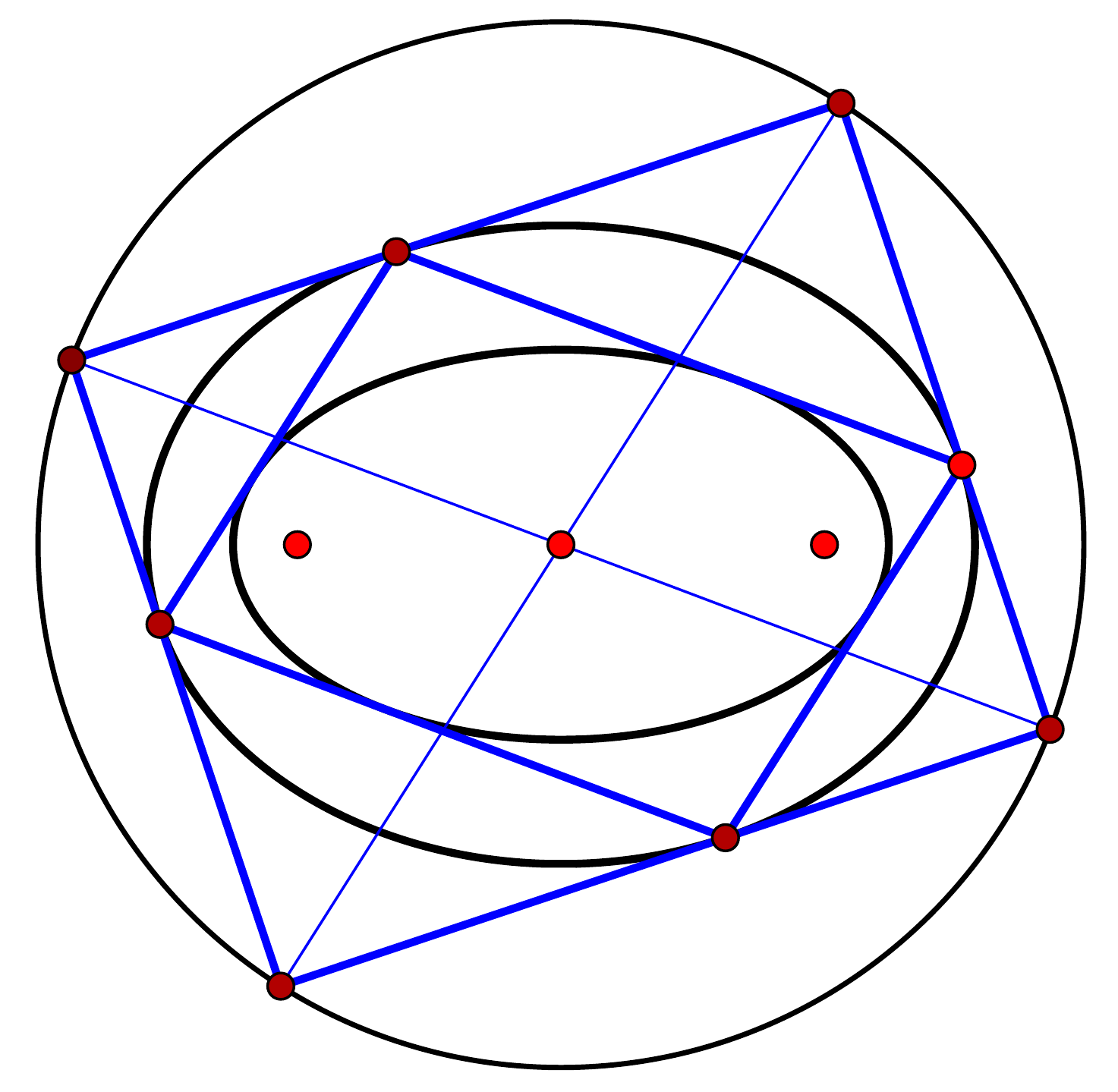}
	\caption{To Theorem \ref{thm:proddist}.}
	\label{quad}
\end{figure}

{\it Second proof}. We make use of  support functions, in accordance with our approach to billiards in this paper. 

As before, consider two ellipses (\ref{eq:pair}), and let $H$ be the support function of the outer one. 
Let $\psi_{i}$ be the angles of the normal at the vertices of the parallelograms, $\delta(\psi)$ the function of reflective angle of the family of the parallelograms.

For a fix a parallelogram, the angles of the normals are 
$$\psi_1,\psi_2,\psi_1+\pi,\psi_2+\pi.$$ 
Denote by $p_1, p_2$ the distances from the origin to the sides of the parallelogram.
By Corollary \ref{cor:sums}, we have in our case
$$
\sin(2\psi_1)=-\sin(2\psi_2).
$$
This implies $\psi_2=\psi_1+\frac{\pi}{2}$. This means that the tangents to $E$ at the vertices of the parallelogram form a rectangle, reproving the observation made in \cite{C-Z}, see Figure \ref{quad}. In addition, 
$$
\delta (\psi_2)+\delta(\psi_1)=\frac{\pi}{2}.
$$

Moreover it follows from equation (\ref{h}) that
$$p_1=H(\psi_1) \cos \delta(\psi_1) +H'(\psi_1) \sin \delta(\psi_1)=H(\psi_2) \cos \delta(\psi_2) -H'(\psi_2) \sin \delta(\psi_2),$$
$$
p_2=H(\psi_1) \cos \delta(\psi_1) -H'(\psi_1) \sin \delta(\psi_1)=H(\psi_2) \cos \delta(\psi_2)+H'(\psi_2) \sin \delta(\psi_2).
$$
This leads to
$$
H(\psi_2)\cos\delta(\psi_2)=H(\psi_1)\cos\delta(\psi_1).
$$
Hence
\begin{equation}\label{delta}
\tan\delta(\psi_1)= \frac{H(\psi_1+\frac{\pi}{2})}{H(\psi_1)}=
\left(\frac{A^2\sin^2\psi_1+B^2\cos^2\psi_1}{A^2\cos^2\psi_1+B^2\sin^2\psi_1}\right)^\frac{1}{2}=\left(\frac{1+ky}{k+y}\right)^
\frac{1}{2},
\end{equation}
where $k:=A^2/B^2=a/b, y:=\tan^2\psi_1.$
From the explicit expressions for $p_1, p_2$ we get:
$$p_1p_2=H(\psi_1)^2\cos^2\delta(\psi_1)-
H'(\psi_1)^2\sin^2\delta(\psi_1).$$
Using Joachimsthal integral, we can substitute in the last formula
$$
H=\frac{1}{J}\sin\delta, H'=\frac{1}{J}\delta'\cos\delta.
$$
We  get
$$
p_1p_2=\frac{1}{J^2}(\sin^2\delta\cos^2\delta-(\delta'\sin\delta\cos\delta)^2)=\frac{1}{J^2}\left(z(1-z)-
\left[\frac{z}{2}'\right]^2\right),$$ where we abbreviate $z:=\sin^2\delta$.

Next we compute $z$ via $y$ using formula (\ref{delta}):
$$                           
z=\frac{1+ky}{(1+k)(1+y)}.
$$
Also,
$$
z'=\frac{k-1}{(k+1)(1+y)^2}\cdot\frac{dy}{d\psi}=\frac{k-1}{(k+1)^2(1+y)^2}\cdot 2\sqrt y(1+y)=\frac{k-1}{(k+1)(1+y)}\cdot 2\sqrt y.
$$

We substitute and get:
$$
p_1p_2=\frac{1}{J^2}\left[\frac{(1+ky)(k+y)}{(1+k)^2(1+y)^2}-
\frac{(k-1)^2y}{(k+1)^2(1+y)^2}\right]=\frac{k}{J^2(k+1)^2}.
$$
To finish the proof, we note that $J=1/(a+b)$ (which is again clear by considering a rectangle that circumscribes the inner ellipse and whose sides are parallel to its axes). This yields $p_1p_2=ab$.
\end{proof}

\end{document}